\documentclass{amsart}

\usepackage{amsmath, amsfonts, mathrsfs}
\usepackage{amssymb,latexsym}
\usepackage{enumerate}
\usepackage{bbm}

\newtheorem{theorem}{Theorem}[section]
\newtheorem{corollary}[theorem]{Corollary}
\newtheorem{lemma}[theorem]{Lemma}

\theoremstyle{definition}

\numberwithin{equation}{section}

\begin{document}

\title{A new form of the Hahn - Banach Theorem}

\author{ Sokol Bush Kaliaj }

\address{
Mathematics Department, 
Science Natural Faculty, 
University of Elbasan,
Elbasan, 
Albania.
}

\email{sokol\_bush@yahoo.co.uk}

\thanks{}



\begin{abstract}
In this paper,  
we present a new form of the Hahn - Banach Theorem in terms of the sub-additive convex functions. 
\end{abstract}

\maketitle

\section{Introduction and Preliminaries}

The Hahn-Banach theorem 
is one of the central results in functional analysis.  
This theorem originated from Hahn \cite{HAN} and Banach \cite{BAN} is of basic importance in the analysis
of problems concerning the existence of continuous linear functionals.
Its principal formulations are 
as a dominated extension theorem (analytic form) and as a separation theorem (geometric form), 
c.f. \cite[Theorem II, 3.2]{SCH} and \cite[Theorem II, 3.1]{SCH}.  
It is known that  \cite[Theorem II, 3.2]{SCH} and \cite[Theorem II, 3.1]{SCH} imply each other.

Throughout this paper, $L$ denotes a vector space over $\Phi$. 
The scalar field $\Phi$ is the real field $\mathbb{R}$ or the complex field $\mathbb{C}$. 
We use the terminologies in the book \cite{SCH}. 
A function $p : L \to \mathbb{R}$  is said to be a \textit{sub-linear function} if  
for every $x, y \in L$ and for every $\lambda \geq 0$, we have
\begin{itemize}
\item
$p(x + y ) \leq p(x) + p(y)$,
\item
$p(\lambda x) = \lambda p(x)$. 
\end{itemize}
The function $p$ is said to be a \textit{semi-norm} if 
$p$ is a sub-linear function such that   
for every $x \in L$ and for every  $\lambda \in \Phi$, we have
$$
p(\lambda x) = |\lambda| p(x). 
$$
A function $\varphi : L \to \mathbb{R}$ is said to be a 
\textit{convex function} if for every $x,y \in L$ and for every $\lambda \in \mathbb{R}$, 
we have
$$
0< \lambda < 1 \Rightarrow 
\varphi(\lambda x + (1-\lambda) y ) \leq 
\lambda \varphi(x) + (1-\lambda) \varphi(y).
$$ 
We say that $\varphi$ is a \textit{sub-additive function} if  
$$
\varphi(x+y) \leq \varphi(x) +\varphi(y), \text{ for all }x, y \in L, 
$$ 
and $\varphi(\theta) =0$, where $\theta$ is the zero vector.   
The function $\varphi$ is called a \textit{sub-additive convex function} if  
$\varphi$ is a sub-additive and convex function.

Let $L$ be a topological vector space with $0$-neighborhood base $\mathfrak{B}$ 
and  let $\varphi : L \to \mathbb{R}$  be a convex function. 
The function $\varphi$ is said to be \textit{upper semi-continuous} at $x_{0} \in L$ 
if given $\varepsilon >0$ there exists a neighborhood $U_{\varepsilon} \in \mathfrak{B}$ 
such that 
$$
x \in x_{0} + U_{\varepsilon} \Rightarrow \varphi(x) < \varphi(x_{0}) + \varepsilon;
$$
$\varphi$ is said to be upper semi-continuous on a subset $G \subset L$ 
if for each $x \in G$, $\varphi$ is upper semi-continuous at $x$.

The following version of the analytic Hahn-Banach theorem 
first appears in Banach \cite{BAN} (see also \cite[Theorem II.3.10]{DUN}). 
\begin{theorem}\label{T_BAN}
Let $\mathfrak{Y}$ be a vector subspace of the real vector space $\mathfrak{X}$ 
and let $p : \mathfrak{X} \to \mathbb{R}$ be a sub-linear function. 
Let $f:\mathfrak{Y} \to \mathbb{R}$ be a linear function with
$$
f(x) \leq p(x), \quad x \in \mathfrak{Y}.
$$
Then there exists a linear function $F:\mathfrak{X} \to \mathbb{R}$ for which 
$$
F(x) =f(x),\quad x \in \mathfrak{Y}; 
\qquad
F(x) \leq p(x),\quad x \in \mathfrak{X}.
$$
\end{theorem}
In this paper, we present a new form of the analytic Hahn-Banach theorem in terms of 
the sub-additive convex functions, see Theorem \ref{T_Main}. 
Since each sub-linear function is sub-additive convex, 
Theorem \ref{T_Main} implies Theorem \ref{T_BAN}. 
It is also shown that  Theorem \ref{T_Main} implies \cite[Theorem II, 3.2]{SCH}, 
see Corollary \ref{C_Main}.

In the papers \cite{NAK1}, \cite{NAK2} and \cite{FOR} are presented generalizations of the Hahn - Banach theorem in terms of the convex functionals. 
There are also generalizations of the Hahn - Banach theorem in the papers 
\cite{HIR}, \cite{CAK}, \cite{PLE}, \cite{RAN} and \cite{SIM}. 
Kakutani \cite{KUK} gave a proof of the analytic Hahn-Banach theorem by
using the Markov-Kakutani fixed-point theorem.

\section{The Main Result}

The main result is Theorem \ref{T_Main}. 
Let us start with two auxiliary lemmas.

\begin{lemma}\label{L.1} 
Let $L$ be a topological vector space over $\mathbb{R}$, 
let $\mathfrak{B}$ be the family of all circled $0$-neighborhood and  
let $\varphi : L \to \mathbb{R}$  be a convex function. 
If $G$ is a non-empty, convex, open subset of $L$ on which $\varphi$ is bounded above, 
then $\varphi$ is upper semi-continuous on $G$. 
\end{lemma}
\begin{proof}
Let $x_{0} \in G$ be and arbitrary vector and let  $0<\varepsilon <1$. 
Then by hypothesis there exists $M>0$ such that 
$$
\varphi(x) \leq M,\text{ for all } x \in G 
$$
and there exists $B \in \mathfrak{B}$ such that $x_{0} + B \subset G$. 
Thus, $\varphi(x) \leq M$ for all $x \in x_{0} + B$. 
If we choose 
$$
\eta = \frac{\varepsilon}{3(1+M)(1+|\varphi(x_{0})|)}
$$
then $\eta B \subset B$, since $B$ is circled. 
For every $u \in B$ the equality
$$
x_{0} + \eta u = \eta(x_{0} + u) + (1-\eta) x_{0}
$$
yields 
\begin{equation*}
\begin{split}
\varphi(x_{0} + \eta u) 
&\leq 
\eta \varphi(x_{0} + u) + (1-\eta)\varphi(x_{0}) \\
&=
\varphi(x_{0}) + \eta \varphi(x_{0} + u) - \eta\varphi(x_{0}) 
\leq \varphi(x_{0}) + \eta M - \eta\varphi(x_{0}) \\
&\leq 
\varphi(x_{0}) + \frac{\varepsilon M }{3(1+M)(1+|\varphi(x_{0})|)}  
+\frac{\varepsilon(-\varphi(x_{0}))}{3(1+M)(1+|\varphi(x_{0})|)}\\
&\leq \varphi(x_{0}) + \frac{\varepsilon}{3} + \frac{\varepsilon}{3} < 
\varphi(x_{0}) + \varepsilon.
\end{split}
\end{equation*}
Thus, for every $x \in x_{0} + \eta B$, we have
\begin{equation*}
\varphi(x) < \varphi(x_{0}) + \varepsilon.
\end{equation*}
This means that $\varphi$ is upper semi-continuous at $x _{0} \in G$, 
and since $x_{0}$ was arbitrary it follows that  
 $\varphi$ is upper semi-continuous on $G$ and the proof is finished.
\end{proof}

\begin{lemma}\label{L.2} 
Let $L$ be a topological vector space over $\mathbb{R}$, 
let $\mathfrak{B}$ be the family of all circled $0$-neighborhood and  
let $\varphi : L \to \mathbb{R}$  be a convex function. 
Then the following statements are equivalent:
\begin{itemize}
\item[(i)]  
$\varphi$ is continuous on $L$.
\item[(ii)]
$\varphi$ is upper semi-continuous on $L$.
\item[(iii)]
There exists a non-empty, convex, open subset of $L$ on which $\varphi$ is bounded above.
\end{itemize}
\end{lemma}
\begin{proof}
$(i) \Leftrightarrow (ii)$ 
Clearly, if $\varphi$ is continuous on $L$, then $\varphi$ is upper semi-continuous on $L$. 
Conversely, assume that  $\varphi$ is upper semi-continuous at $x_{0} \in L$. 
Then, given $\varepsilon >0$ there exists $U_{\varepsilon} \in \mathfrak{B}$ such that 
\begin{equation}\label{eq_LDD.1}
u \in U_{\varepsilon} = -U_{\varepsilon}  
\Rightarrow 
\max \{\varphi(x_{0} + u), \varphi(x_{0} - u)\} < \varphi(x_{0}) + \varepsilon
\end{equation}
and since 
\begin{equation*}
2\varphi(x_{0}) \leq \varphi(x_{0} + u) + \varphi(x_{0} - u)
\end{equation*}
it follows that
\begin{equation}\label{eq_LDD.2} 
u \in U_{\varepsilon}  
\Rightarrow 2\varphi(x_{0}) \leq \varphi(x_{0} + u) + \varphi(x_{0} - u) 
< 
2\varphi(x_{0}) + 2\varepsilon. 
\end{equation}
The last result together with \eqref{eq_LDD.1} yields 
\begin{equation*}
u \in U_{\varepsilon}  
\Rightarrow  
\varphi(x_{0}) - \varepsilon < \varphi(x_{0} + u) < \varphi(x_{0}) + \varepsilon,
\end{equation*}
since $\varphi(x_{0} + u) \leq \varphi(x_{0}) - \varepsilon$ together with 
$\varphi(x_{0} - u) \leq \varphi(x_{0}) + \varepsilon$ yields 
$$
\varphi(x_{0} + u) + \varphi(x_{0} - u) < 2\varphi(x_{0})
$$
contradicting \eqref{eq_LDD.2}. 
This means that $\varphi$ is continuous at $x_{0}$.

$(ii) \Leftrightarrow (iii)$ 
Assume that $\varphi$ is upper semi-continuous on $L$ and let 
$G = \{x \in L : \varphi(x) < 1 \}$. 
It is easy to see that $G$ is non-empty convex subset of $L$. 
Let $x_{0} \in G$ be an arbitrary vector and let $\varepsilon = \frac{1-\varphi(x_{0})}{2}$. 
Then there exists $U_{\varepsilon} \in \mathfrak{B}$ such that
\begin{equation*}
\begin{split}
x \in x_{0} + U_{\varepsilon} \Rightarrow \varphi(x)< \varphi(x_{0})  + \varepsilon < 1 
\Rightarrow x \in G.
\end{split}
\end{equation*} 
Thus, $ x_{0} + U_{\varepsilon} \subset G$, 
and since $x_{0}$ was arbitrary it follows that $G$ is an open subset of $L$.

Conversely, assume that $G$ is a non-empty, convex, open subset of $L$ on which $\varphi$ is bounded above 
and let $x_{0} \in L$ be an arbitrary vector. 
Then, there exists $M>0$ such that 
$\varphi(x) \leq M$ for all $x \in G$. 
If $x_{0} \in G$, then by Lemma \ref{L.1} $\varphi$ is upper semi-continuous at $x_{0}$.

It remains to consider $x_{0} \not\in G$. 
Fix a vector $y_{0} \in G$ and a real positive number $\rho >1$ 
and define $w = y_{0} + \rho(x_{0} - y_{0})$.   
The function $h : L \to L$ defined by
\begin{equation*}
h(y) = \frac{\rho-1}{\rho}y + \frac{1}{\rho}w, \text{ for all }y \in L
\end{equation*}
is a homeomorphism, c.f. \cite[1.1, p.13]{SCH}. 
The function $h$ transforms $y_{0}$ into $x_{0}$ and $G$ into an open and convex set $h(G)$ 
containing $x_{0}$. 
For every $x \in h(G)$, we have 
$$
x = h(h^{-1}(x)) = \frac{\rho-1}{\rho}h^{-1}(x) + \frac{1}{\rho}w 
$$  
and  
\begin{equation*}
\begin{split}
\varphi(x) 
&= \varphi \left ( \frac{\rho-1}{\rho}h^{-1}(x) + \frac{1}{\rho}w \right ) 
\leq \frac{\rho-1}{\rho} \varphi(h^{-1}(x)) + \frac{1}{\rho} \varphi(w) \\
&\leq \frac{\rho-1}{\rho} M + \frac{1}{\rho} \varphi(w) = M'<+\infty.
\end{split}
\end{equation*}
Thus, $\varphi$ is bounded over $h(G)$. 
Therefore, by Lemma \ref{L.2}, $\varphi$ is upper semi-continuous at $x_{0} \in h(G)$. 
Since $x_{0}$ was arbitrary $\varphi$ is upper semi-continuous at every $x \in G$ and 
at every $x \not\in G$. 
Thus, $\varphi$ is upper semi-continuous on $L$, 
and this ends the proof. 
\end{proof}

We are now ready to present the main result.

\begin{theorem}\label{T_Main}  
Let $M$ be a subspace of a vector space $L$ over $\mathbb{R}$ and 
let $\varphi : L \to \mathbb{R}$  be a sub-additive convex function. 
If $g: M \to \mathbb{R}$ is a linear function such that 
$g$ "dominated" by $\varphi$, i.e.,  
$$
(\forall x \in M)[g(x) \leq \varphi(x)],
$$
then there exists a linear function $f: L \to \mathbb{R}$ such that
\begin{itemize}
\item[(i)]
$f$ extends $g$ to $L$, i.e., 
$$
f \vert_{M} = g,
$$
\item[(ii)]
$f$ "dominated" by $\varphi$, i.e., 
$$
(\forall x \in L)[f(x) \leq \varphi(x)].
$$
\end{itemize}
\end{theorem}
\begin{proof}
Define
$$
U_{n} = \left \{ x \in L : \varphi(x) < \frac{1}{n} \right \}, n \in \mathbb{N}
$$
and
\begin{equation*} 
\mathfrak{B} = 
\left \{ V_{n} \in 2^{L} : V_{n} = U_{n} \cap (-U_{n}), n \in \mathbb{N} \right \}, 
\end{equation*}
where
$$
-U_{n} = \left \{ -x \in L : \varphi(x) < \frac{1}{n} \right \} = \left \{ x \in L : \varphi(-x) < \frac{1}{n} \right \}.
$$
The family $\mathfrak{B}$ is a filter base in $L$, since 
$$
V_{n} \cap V_{m} \supset V_{m}, \quad m > n
$$
and $\theta  \in V_{n}$ for all $n \in \mathbb{N}$. 

Note that $(-1)V_{n} = V_{n}$ and $0 V_{n} = \{\theta \} \subset V_{n}$; 
given $0<\lambda <1$ or $-1<\lambda <0$, we have also
$$
\lambda V_{n} \subset V_{n}, 
$$
since
\begin{equation*}
\begin{split}
\varphi(x) 
&= \varphi \left ( \frac{\lambda}{\lambda} x \right ) 
= \varphi \left ( \lambda \frac{1}{\lambda} x  + (1 - \lambda) \theta \right ) 
\leq \lambda \varphi \left (  \frac{1}{\lambda} x \right ).
\end{split}
\end{equation*}
Thus, $V_{n}$ is circled. 
Let us show next that $V_{n}$ is radial. Since $V_{n}$ is circled, 
it is enough to show that given $x_{0} \in L$ there exists $\mu \in \mathbb{R}$ such that 
$x_{0} \in \mu V_{n}$.  
If $x_{0} \in V_{n}$ then $\mu=1$. 
Otherwise $x_{0} \not\in U_{n}$ or  $x_{0} \not\in -U_{n}$; 
$\varphi(x_{0}) \geq \frac{1}{n}$ or $\varphi(-x_{0}) \geq \frac{1}{n}$. 
There exists $\mu >1$ such that 
$\frac{1}{\mu}\varphi(x_{0}) < \frac{1}{n}$ and 
$\frac{1}{\mu}\varphi(-x_{0}) < \frac{1}{n}$. 
Since
\begin{equation*}
\begin{split}
\varphi \left ( \frac{1}{\mu}x_{0} \right ) 
=&  \varphi \left ( \frac{1}{\mu}x_{0} + \left ( 1-\frac{1}{\mu} \right )\theta \right ) 
\leq \frac{1}{\mu} \varphi(x_{0}) < \frac{1}{n},\\
\varphi \left ( -\frac{1}{\mu}x_{0} \right ) 
=&  \varphi \left ( \frac{1}{\mu}(-x_{0}) + \left ( 1-\frac{1}{\mu} \right )\theta \right ) 
\leq \frac{1}{\mu} \varphi(-x_{0}) < \frac{1}{n},
\end{split}
\end{equation*}
it follows that $x_{0} \in \mu U_{n} \cap \mu (-U_{n}) = \mu V_{n}$. 
Thus, $\mathfrak{B}$ is a filter base in $L$ such that for each $n \in \mathbb{N}$, we have
\begin{itemize}
\item[(a)]
$V_{2n} + V_{2n} \subset V_{n}$, since $\varphi$ is sub-additive function,
\item[(b)]
$V_{n}$ is radial and circled. 
\end{itemize} 
Therefore, by \cite[I, 2.1, p.15]{SCH} the family $\mathfrak{B}$ 
is a $0$-neighborhood base for a unique topology $\mathfrak{T}$ 
under which $L$ is a topological vector space.

Given any vector $x_{0} \in U_{1}$, there exists $n \in \mathbb{N}$ such that 
$\frac{1}{n} < 1 - \varphi(x_{0})$. Then, $x_{0} + V_{n} \subset U_{1}$. 
This means that $U_{1}$ is an open set, 
and since $U_{1}$ is also convex, by Lemma \ref{L.2} 
it follows that $\varphi$  is continuous on $L$.

We now consider $L \times \mathbb{R}$ as the product of two topological vector spaces. 
Note that if $T : L \times \mathbb{R} \to \mathbb{R}$ is a linear function, then 
\begin{equation*}
T(x,t) = h(x) + \alpha t, \text{ for all } (x,t) \in L \times \mathbb{R}, 
\end{equation*}
where 
$h(x) = T(x, 0)$ for all $x \in L$   
and, $\alpha = T(0, 1)$. 
Hence,
\begin{equation*}
H = \{ (x,t) \in M \times \mathbb{R} : g(x) - t = 1 \}
\end{equation*}
is a linear manifold in the subspace $M \times \mathbb{R}$.  
Since $(x,t) \to S(x,t) = g(x) - t$ is a non-zero linear functional,  
by \cite[I, 4.1]{SCH} 
$H$ is a hyperplane in $M \times \mathbb{R}$ 
and a linear manifold in $L \times \mathbb{R}$.

Since $\varphi$ is a continuous function on $L$ it follows that 
$$
G = \{ (x,t) \in L \times \mathbb{R} : \varphi(x) - t < 1 \}
$$
is an open set, 
and since $g(x) \leq \varphi(x)$ for $x \in M$ 
it follows that $G \cap H = \emptyset$;  
$G$ is also a non-empty convex set.
Therefore, by \cite[Theorem II, 3.1]{SCH} there exists a closed hyperplane 
$H_{1}$ in $L \times \mathbb{R}$ such that 
$H_{1} \supset H$ and $H_{1} \cap G = \emptyset$. 
Since $H_{1}$ is a hyperplane, $H_{1} = (x_{0}, t_{0}) + H_{0}$ 
where $H_{0}$ is a maximal subspace of $L \times \mathbb{R}$. 
If we choose $(x_{0}, t_{0}) \in H \subset H_{1} \cap M \times \mathbb{R}$, then 
$$
H_{1} \cap (M \times \mathbb{R}) = 
(x_{0}, t_{0}) + (H_{0} \cap (M \times \mathbb{R})), 
$$
and since
$$
H_{1} \cap (M \times \mathbb{R}) \neq (M \times \mathbb{R}),
\quad 
((0,0) \not\in H_{1}\text{ and }(0,0) \in (M \times \mathbb{R}) )
$$ 
it follows that 
$H_{1} \cap (M \times \mathbb{R})$ is a hyperplane in $M \times \mathbb{R}$. 
Further, $H_{1} \cap (M \times \mathbb{R}) \supset H$ 
implies
\begin{equation}\label{eq_Dom.1}
H_{1} \cap (M \times \mathbb{R}) = H =(x_{0}, t_{0}) + (H_{0} \cap (M \times \mathbb{R}));
\end{equation}
($M \times \mathbb{R}$ is the algebraic sum of subspaces 
$H_{0} \cap (M \times \mathbb{R})$ and 
$\{ \lambda (x_{0}, t_{0}) : \lambda \in \mathbb{R} \}$). 
By \cite[I,4.1]{SCH},  
there exits a non-zero linear functional $F:L \times \mathbb{R} \to \mathbb{R}$ such that 
$H_{1} = \{ (x,t) \in L \times \mathbb{R}: F(x,t) = 1 \}$. 
The least result together with \eqref{eq_Dom.1} yields  
$$
F(x,0) = S(x,0) = g(x), \text{ for all }x \in M;
$$
that is, the linear function $f(\cdot) = F(\cdot,0)$ 
is an extension of $g$ to $L$. 

It remains to prove that $f$ "dominated" by $\varphi$. Note that 
$$
F(x,t) = f(x) -t,\text{ for all }(x,t) \in L \times \mathbb{R}, 
$$
since $(0,-1) \in H \subset H_{1}$. 
By $H_{1} \cap G = \emptyset$ it follows that
\begin{equation}\label{eq_DD1}
f(x)= 1+t \leq \varphi(x), \text{ for all }(x,t) \in H_{1}.  
\end{equation}
While $(x,t) \not\in H_{1}$ implies that there exists a real number $r \neq 1$ such that 
$$
f(x) - t = r.
$$
Then, $(x, t + r -1 ) \in H_{1}$ and therefore \eqref{eq_DD1} yields
$$
f(x) = 1 + (t + r -1)  \leq \varphi(x).
$$
Thus, $f$ "dominated" by $\varphi$ and this ends the proof.
\end{proof}

\begin{corollary}\label{C_Main}
Let $M$ be a subspace of a vector space $L$ over $\mathbb{C}$ and 
let $p : L \to [0,+\infty)$  be a semi-norm. 
If $g: M \to \mathbb{C}$ is a linear functional such that 
$g$ "dominated" by $p$, i.e., 
$$
(\forall x \in M)[|g(x)| \leq p(x)],
$$
then there exists a linear functional $f: L \to \mathbb{C}$ such that
\begin{itemize}
\item
$f$ extends $g$ to $L$, i.e., 
$$
f \vert_{M} = g,
$$
\item
$f$ "dominated" by $p$, i.e., 
$$
(\forall x \in L)[|f(x)| \leq p(x)].
$$
\end{itemize}
\end{corollary}
\begin{proof}
There exists a linear functional $g_{0} : M \to \mathbb{R}$ such that
\begin{equation*}
g(x) = g_{0}(x) - i g_{0}(ix), \text{ for all } x \in M. 
\end{equation*}
We have also that $p$ is a convex functional such that $p(\lambda x) = \lambda p(x)$ 
for every $\lambda \geq 0$ and $x \in L$. 
Thus, regarding $L$ as a real linear space, and applying Theorem \ref{T_Main}, 
a real linear function $f_{0} : L \to \mathbb{R}$ is obtained for which 
$f_{0}\vert_{M} = g_{0}$ and $f_{0}(x) \leq p(x)$ for all $x \in L$.
Let the function $f : L \to \mathbb{C}$ be defined by the 
equation 
\begin{equation*}
f(x) = f_{0}(x) - i f_{0}(i x).
\end{equation*}
Clearly, $f$ is a linear functional and
$$
f(x) = f_{0}(x) - i f_{0}(i x) = g_{0}(x) - i g_{0}(i x)= g(x), 
\text{ for all }x \in M.
$$ 
Thus $f$ is an extension of $g$. 
Finally, let $f(x) = |f(x)| e^{i \theta(x)}$, 
where $\theta(x)$ is an argument of the complex number $f(x)$; 
then
$$
|f(x)| = f_{0} \left ( \frac{x}{e^{i \theta(x)}} \right ) 
\leq 
p \left ( \frac{x}{e^{i \theta(x)}} \right ) = p(x),
$$ 
which proves that $|f( \cdot )| \leq p( \cdot )$ 
and this ends the proof.
\end{proof}

Since \cite[Theorem II, 3.1]{SCH} implies Theorem \ref{T_Main} and 
Theorem \ref{T_Main} implies  \cite[Theorem II, 3.2]{SCH}, 
we conclude that Theorem \ref{T_Main},  
\cite[Theorem II, 3.2]{SCH} and \cite[Theorem II, 3.1]{SCH} imply each other.

\bibliographystyle{plain}

\end{document}